\documentclass[11pt,reqno]{amsart}
\usepackage{microtype}

\usepackage[dvipsnames]{xcolor}
\usepackage
[colorlinks=true,linkcolor=Maroon,citecolor=OliveGreen]
{hyperref}
\usepackage{amsmath, amsfonts, amsthm, amssymb, bm}
\usepackage[shortlabels]{enumitem}
\setlist[enumerate]{label={(\arabic*)}}
\usepackage[capitalize]{cleveref}
\crefname{equation}{}{}

\usepackage[numbers]{natbib}

\newtheorem{theorem}{Theorem}
\newtheorem{lemma}[theorem]{Lemma}

\newtheorem{corollary}[theorem]{Corollary}

\newtheorem*{theorem*}{Theorem}

\newtheorem*{conjecture}{Conjecture}

\theoremstyle{definition}

\theoremstyle{remark}

\def\R{\mathbf{R}}

\newcommand\opr[1]{\operatorname{#1}}

\newcommand{\pp}{\mathbb{P}}
\newcommand{\ee}{\mathbb{E}}
\def\half{{\tfrac12}}

\def\KL{\opr{KL}}
\newcommand{\norm}[1]{\left\lVert#1\right\rVert}

\numberwithin{equation}{section}

\author{Anay Aggarwal}
\address{\parbox{\linewidth}{Stanford University, CA, USA\vspace{0.1cm}}}
\email{anayagg@stanford.edu}

\begin{document}
\title[Shifted R\'enyi Divergence]{Resolution of a Conjecture on Shifted R\'enyi Divergence}

\begin{abstract}
Altschuler and Chewi \cite{altschuler2023fasterhighaccuracylogconcavesampling} conjecture that the sub-Gaussian Orlicz--Wasserstein shifted R\'enyi divergence between two isotropic Gaussians of equal covariance is attained by a deterministic shift of the mean, which would give an exact closed form for the shift budget consumed by their analysis. We show that this conjecture is true when $q=1$, where the R\'enyi divergence degenerates to the Kullback--Leibler divergence, and false for every $q>1$ outside the degenerate regime.
\end{abstract}

\maketitle

\section{Introduction}

In \cite{altschuler2023fasterhighaccuracylogconcavesampling}, the authors make a conjecture that, if true, would provide exact closed-form bounds in their analysis used to solve a longstanding open problem in MCMC complexity. This is listed as Conjecture B.1 in the appendix of the paper. In this note we identify exactly when the conjecture is true and false. Before stating the conjecture, let us build up the setting of the problem.

Let $(\Omega,\mathcal{F},\pp)$ be a probability space and $(E,\mathcal{E})$ a measurable space. We write $\mu\ll \nu$ for probability measures $\mu,\nu$ if every $\nu$-null set is also $\mu$-null. If $\mu\ll \nu$ then the Radon--Nikodym derivative $f=d\mu/d\nu$ is well defined. When $\Pi(\mu,\nu)$ is the set of couplings of $\mu$ and $\nu$, the $p$-Wasserstein distance is defined as
$$W_p(\mu,\nu):=\left(\inf_{\pi\in \Pi(\mu,\nu)}\int_{E\times E}\norm{x-y}^p\,\pi(dx,dy)\right)^{1/p}.$$
It is helpful to generalize this definition to the Orlicz--Wasserstein distance. Let $\psi:[0,\infty)\to[0,\infty)$ be an Orlicz function, i.e.\ $\psi$ is convex, $\psi(0)=0$, $\psi(t)>0$ for $t>0$, and $\psi(t)\to\infty$ as $t\to\infty$. Then the Orlicz norm is
$$\norm{Z}_{\psi}:=\inf\{\lambda>0:\ee\,\psi(\norm{Z}/\lambda)\le 1\},$$
where $Z$ is a random variable over $\R^d$ and $\norm{\bullet}$ is the Euclidean norm. Then the Orlicz--Wasserstein distance is
$$W_{\psi}(\mu,\nu)=\inf\{\norm{X-Y}_{\psi}:(X,Y)\sim \pi\in \Pi(\mu,\nu)\}.$$
This allows us to generalize the R\'enyi divergence
$$R_q(\mu\|\nu)=\frac{1}{q-1}\log \int_E \left(\frac{d\mu}{d\nu}\right)^q d\nu$$
to the \emph{shifted R\'enyi divergence}
$$R^{(w)}_{q,\psi}(\mu\|\nu)=\inf\{R_q(\mu'\|\nu):W_{\psi}(\mu,\mu')\le w\}$$
for some $w\ge 0$ and Orlicz function $\psi$. As usual, $R_1$ is defined by continuity and equals the Kullback--Leibler divergence $\KL$. The conjecture is on a closed form for this shifted divergence for two normals in $\R^d$ with the same covariance matrix.

\begin{conjecture}[\cite{altschuler2023fasterhighaccuracylogconcavesampling}, Conjecture B.1]\label{conj:main}
Let $\psi_2(t)=e^{t^2}-1$, $\mu=N(x,\sigma^2I_d)$, $\nu=N(0,\sigma^2 I_d)$, and set
$$s=w\sqrt{\log 2},\qquad m=\norm{x},\qquad c^*=\max(0,1-s/m).$$
Then for all $q\ge 1$,
$$R^{(w)}_{q,\psi_2}(\mu\|\nu)=R_q\big(N(c^*x,\sigma^2I_d)\,\big\|\,N(0,\sigma^2I_d)\big)=\frac{q(m-s)_+^2}{2\sigma^2}.$$
\end{conjecture}

The $\le$ direction of this identity is trivial: a deterministic vector $v$ has $\norm{v}_{\psi_2}=\norm{v}/\sqrt{\log 2}$, so the shift carrying $x$ to $c^*x$ has Orlicz norm exactly $\min(w,m/\sqrt{\log 2})$ and is thus feasible, and its cost is computed by \cref{lem:gauss} below. It is the $\ge$ direction that requires justification. In this note we prove that the identity holds for $q=1$ and fails for every $q>1$ in the nondegenerate regime $0<s<m$. For the rest of the paper we abbreviate $\psi_2$ as $\psi$ and write $\mu_0:=(m-s)_+$, and we drop $\psi$ from the subscript of the shifted divergence.

We record the two standard computations we use. The first is the equal-covariance Gaussian formula, and the second is the general scalar formula of van Erven and Harremo\"es \cite{van_Erven_2014}, which we need because our counterexample perturbs the variance.

\begin{lemma}\label{lem:gauss}
For $q\ge 1$ and $m_1,m_2\in\R^d$,
$$R_q\big(N(m_1,\sigma^2I_d)\,\big\|\,N(m_2,\sigma^2I_d)\big)=\frac{q\norm{m_1-m_2}^2}{2\sigma^2}.$$
\end{lemma}

\begin{lemma}\label{lem:veh}
Let $q>1$ and let $\sigma_q^2:=(1-q)\sigma_1^2+q\sigma_2^2$. If $\sigma_q^2>0$ then
$$R_q\big(N(\mu_1,\sigma_1^2)\,\big\|\,N(\mu_2,\sigma_2^2)\big)=\frac{q(\mu_1-\mu_2)^2}{2\sigma_q^2}+\frac{1}{1-q}\log\frac{\sigma_q}{\sigma_1^{1-q}\sigma_2^{q}},$$
and the divergence is $+\infty$ if $\sigma_q^2\le 0$. Moreover $R_q$ is additive across independent product factors.
\end{lemma}

\section{The case \texorpdfstring{$q=1$}{q=1}}

The positive result uses nothing about $\psi_2$ beyond $\psi_2^{-1}(1)=\sqrt{\log 2}$, so we state it for a general Orlicz function. The only input on the transport side is the following observation, which says that the deterministic shift is exactly extremal for the mean displacement over an Orlicz ball.

\begin{lemma}\label{lem:jensen}
Let $\psi$ be an Orlicz function and let $D$ be an $\R^d$-valued random variable with $\norm{D}_{\psi}\le w$. Then $\ee\norm{D}\le w\,\psi^{-1}(1)$.
\end{lemma}

\begin{proof}
By definition of the Orlicz norm and right-continuity of $\lambda\mapsto \ee\,\psi(\norm{D}/\lambda)$ we have $\ee\,\psi(\norm{D}/w)\le 1$. Set $V:=\psi(\norm{D}/w)$, so that $\norm{D}=w\,\psi^{-1}(V)$ and $\ee V\le 1$. Since $\psi$ is convex with $\psi(0)=0$, its inverse $\psi^{-1}$ is concave on $[0,\infty)$, so Jensen gives
$$\ee\norm{D}=w\,\ee\,\psi^{-1}(V)\le w\,\psi^{-1}(\ee V)\le w\,\psi^{-1}(1),$$
using monotonicity of $\psi^{-1}$ in the last step.
\end{proof}

For $\psi=\psi_2$ this reads $\ee\norm{D}\le w\sqrt{\log 2}=s$, and one can see the mechanism concretely: $V=e^{\norm{D}^2/w^2}-1$, $\psi_2^{-1}(v)=\sqrt{\log(1+v)}$, and concavity of $v\mapsto\sqrt{\log(1+v)}$ is what forbids a random displacement from purchasing more mean than a constant one.

\begin{theorem}\label{thm:q1}
The conjecture is true for $q=1$. More generally, for any Orlicz function $\psi$, with $s:=w\,\psi^{-1}(1)$,
$$R^{(w)}_{1,\psi}\big(N(x,\sigma^2I_d)\,\big\|\,N(0,\sigma^2I_d)\big)=\frac{(m-s)_+^2}{2\sigma^2},$$
and the infimum is attained at the deterministic shift $N(c^*x,\sigma^2I_d)$.
\end{theorem}

\begin{proof}
The upper bound is the feasibility of the deterministic shift together with \cref{lem:gauss}, so it suffices to settle the lower bound. We may assume $\mu_0=m-s>0$, else there is nothing to prove. Write $u=x/m$ and fix $\theta:=\mu_0/\sigma^2>0$.

Let $\mu'$ satisfy $W_{\psi}(\mu,\mu')\le w$; we may assume $\mu'\ll\nu$, else $\KL(\mu'\|\nu)=\infty$. Fix $\varepsilon>0$. Since $W_\psi$ is an infimum over couplings, which need not be attained, there is a coupling $(X,X')$ with $X\sim\mu$, $X'\sim\mu'$ and $\norm{D}_{\psi}\le w+\varepsilon$ for $D:=X-X'$. By \cref{lem:jensen}, $\ee\norm{D}\le (w+\varepsilon)\psi^{-1}(1)=s+\varepsilon\psi^{-1}(1)$; in particular $\ee\norm{D}<\infty$, so $\ee_{\mu'}\langle u,y\rangle$ is well defined and finite, and by Cauchy--Schwarz
\begin{equation}\label{eq:mean}
\ee_{\mu'}\big[\langle u,y\rangle\big]=\langle u,x\rangle-\ee\big[\langle u,D\rangle\big]\ \ge\ m-\ee\norm{D}\ \ge\ \mu_0-\varepsilon\psi^{-1}(1).
\end{equation}
Apply the Gibbs variational principle with the linear test function $f(y)=\theta\langle u,y\rangle$, whose exponential moment under $\nu$ is finite with $\log\ee_\nu e^{f}=\theta^2\sigma^2/2$:
$$\KL(\mu'\|\nu)\ \ge\ \ee_{\mu'}[f]-\log\ee_{\nu}e^{f}\ =\ \theta\,\ee_{\mu'}\big[\langle u,y\rangle\big]-\frac{\theta^2\sigma^2}{2}.$$
Combining with \eqref{eq:mean} and letting $\varepsilon\downarrow 0$ gives
$$\KL(\mu'\|\nu)\ \ge\ \theta\mu_0-\frac{\theta^2\sigma^2}{2}=\frac{\mu_0^2}{2\sigma^2},$$
by the choice $\theta=\mu_0/\sigma^2$, which maximizes the right-hand side. Taking the infimum over feasible $\mu'$ finishes the proof.
\end{proof}

\section{The case \texorpdfstring{$q>1$}{q>1}}

We construct a family of counterexamples $\mu_{\beta}'$ parametrized the parameter $\beta$. Fix $q>1$, $\sigma^2>0$, $w>0$ and $x\in\R^d$ with $0<s<m$, and let $u=x/m$. Let $X\sim\mu=N(x,\sigma^2I_d)$ and consider displacements of the form
\begin{equation}\label{eq:family}
D_\beta=\big(\alpha+\beta\langle u,X\rangle\big)u,\qquad \beta\in[0,\beta_{\max}),\quad \beta_{\max}:=\frac{w}{\sigma\sqrt2},
\end{equation}
with $\alpha=\alpha(\beta)\in\R$ to be chosen, and set our counterexample $\mu'_\beta=\opr{law}(X-D_\beta)$, so that $W_\psi(\mu,\mu'_\beta)\le \norm{D_\beta}_\psi$. Since $\langle u,X\rangle\sim N(m,\sigma^2)$, the scalar $\langle u,D_\beta\rangle$ is Gaussian with
$$\mu_D=\alpha+\beta m,\qquad \tau=\beta\sigma,$$
and $\norm{D_\beta}=|\langle u,D_\beta\rangle|$. We begin the analysis with a quick lemma.

\begin{lemma}\label{lem:budget}
Let $D$ be a scalar $N(a,\tau^2)$ variable with $2\tau^2<w^2$. Then
$$\ee\exp(D^2/w^2)=\Big(1-\frac{2\tau^2}{w^2}\Big)^{-1/2}\exp\Big(\frac{a^2}{w^2-2\tau^2}\Big),$$
so $\ee\exp(D^2/w^2)=2$ iff
\begin{equation}\label{eq:muD}
a^2=(w^2-2\tau^2)\Big(\log 2+\half\log\Big(1-\frac{2\tau^2}{w^2}\Big)\Big).
\end{equation}
\end{lemma}

\begin{proof}
With $c=1/w^2$,
$$\ee e^{cD^2}=\int \frac{e^{cd^2}}{\sqrt{2\pi\tau^2}}e^{-(d-a)^2/(2\tau^2)}\,dd
=(1-2c\tau^2)^{-1/2}\exp\Big(\frac{ca^2}{1-2c\tau^2}\Big),$$
valid exactly when $2c\tau^2<1$. Taking logarithms and solving for $a^2$ gives the desired result.
\end{proof}

Define $\mu_D(\beta)\ge0$ by \eqref{eq:muD} with $\tau=\beta\sigma$, i.e.
\begin{equation}\label{eq:muDbeta}
\mu_D(\beta)^2=(w^2-2\beta^2\sigma^2)\Big(\log 2+\half\log\Big(1-\frac{2\beta^2\sigma^2}{w^2}\Big)\Big),
\end{equation}
which is positive for $\beta$ in a neighbourhood of $0$ since $\mu_D(0)^2=w^2\log 2=s^2>0$, and set $\alpha(\beta):=\mu_D(\beta)-\beta m$. With this choice $\norm{D_\beta}_\psi=w$ exactly, so $\mu'_\beta$ is feasible. The crucial point of \eqref{eq:muDbeta} is the following corollary:

\begin{corollary}\label{cor:even}
$\beta\mapsto\mu_D(\beta)^2$ is an even, real-analytic function of $\beta$ near $0$ with $\mu_D(0)=s$. Consequently $\mu_D$ is even near $0$ and $\mu_D'(0)=0$; in fact $\mu_D(\beta)=s-O(\beta^2)$.
\end{corollary}

We may now explicitly compute the divergence.

\begin{lemma}\label{lem:pushforward}
For $\beta\in[0,1)$ with $\mu_D(\beta)$ as above, $\mu'_\beta$ is Gaussian with
$$\langle u,\cdot\rangle\text{-marginal }=N\big(m-\mu_D(\beta),\,(1-\beta)^2\sigma^2\big),$$
and with all marginals orthogonal to $u$ equal to those of $\mu$, namely $N(0,\sigma^2)$. Hence, writing $\sigma_q^2(\beta):=q\sigma^2-(q-1)(1-\beta)^2\sigma^2$ and $\mu_1(\beta):=m-\mu_D(\beta)$,
\begin{equation}\label{eq:J}
J(\beta):=R_q(\mu'_\beta\|\nu)=\frac{q\,\mu_1(\beta)^2}{2\sigma_q^2(\beta)}+\frac{1}{1-q}\log\frac{\sigma_q(\beta)}{\big((1-\beta)\sigma\big)^{1-q}\sigma^{q}}.
\end{equation}
\end{lemma}

\begin{proof}
Since $D_\beta$ is a multiple of $u$, only the $u$-component of $X$ is altered:
$$\langle u,X-D_\beta\rangle=(1-\beta)\langle u,X\rangle-\alpha\sim N\big((1-\beta)m-\alpha,(1-\beta)^2\sigma^2\big),$$
and $(1-\beta)m-\alpha=m-\mu_D(\beta)$. The components of $X$ orthogonal to $u$ are independent of $\langle u,X\rangle$ and untouched, and since $x=mu$ they are already distributed as $N(0,\sigma^2)$, matching $\nu$. By additivity of $R_q$ over independent factors (\cref{lem:veh}) the orthogonal directions contribute $0$, and \eqref{eq:J} is \cref{lem:veh} applied to the $u$-marginal. Note $\sigma_q^2(0)=\sigma^2>0$, so \eqref{eq:J} is finite for small $\beta$.
\end{proof}

To finish, we can show that $J$ is decreasing, because at $\beta=0$ we get
\begin{equation}\label{eq:J0}
J(0)=\frac{q\,\mu_0^2}{2\sigma^2},
\end{equation}
which yields the desired result.

\begin{theorem}\label{thm:qbig}
The conjecture is false for every $q>1$ whenever $0<s<m$. Indeed, with $J$ as in \eqref{eq:J},
$$J'(0)=-\frac{q(q-1)\mu_0^2}{\sigma^2}<0,$$
so that for all sufficiently small $\beta>0$,
$$R^{(w)}_{q}\big(N(x,\sigma^2I_d)\,\big\|\,N(0,\sigma^2I_d)\big)\ \le\ J(\beta)\ <\ \frac{q(m-s)^2}{2\sigma^2}.$$
\end{theorem}

\begin{proof}
By \cref{lem:pushforward} each $\mu'_\beta$ is feasible, so $R^{(w)}_q(\mu\|\nu)\le J(\beta)$ for all small $\beta\ge0$; in view of \eqref{eq:J0} it suffices to compute $J'(0)$ and check that it is negative. Write $v(\beta):=(1-\beta)^2\sigma^2$ for the variance of the $u$-marginal of $\mu'_\beta$, so that $\sigma_q^2=q\sigma^2-(q-1)v$ and
$$v(0)=\sigma^2,\qquad v'(0)=-2\sigma^2,\qquad \sigma_q^2(0)=\sigma^2,\qquad \big(\sigma_q^2\big)'(0)=-(q-1)v'(0)=2(q-1)\sigma^2.$$
Differentiate the two terms of \eqref{eq:J} separately.

For the mean term, $\mu_1'(0)=-\mu_D'(0)=0$ by \cref{cor:even}, so only the denominator moves:
$$\frac{d}{d\beta}\bigg[\frac{q\mu_1^2}{2\sigma_q^2}\bigg]_{\beta=0}=-\frac{q\mu_0^2}{2\sigma^4}\big(\sigma_q^2\big)'(0)=-\frac{q\mu_0^2}{2\sigma^4}\cdot 2(q-1)\sigma^2=-\frac{q(q-1)\mu_0^2}{\sigma^2}.$$

For the log term, write it as
$$\frac{1}{1-q}\bigg[\half\log\sigma_q^2-\frac{1-q}{2}\log v-\frac{q}{2}\log\sigma^2\bigg],$$
whose derivative at $\beta=0$ is
$$\frac{1}{1-q}\cdot\half\bigg[\frac{(\sigma_q^2)'(0)}{\sigma_q^2(0)}-(1-q)\frac{v'(0)}{v(0)}\bigg]
=\frac{1}{2(1-q)}\Big[2(q-1)-(1-q)(-2)\Big]=0.$$
Equivalently: as a function of $v$ the log term has derivative $\half\big(1/\sigma_q^2-1/v\big)$, which vanishes at $v=\sigma^2$, so the log-determinant contribution is stationary at the deterministic shift and only enters at second order.

Adding the two gives $J'(0)=-q(q-1)\mu_0^2/\sigma^2$, which is strictly negative because $q>1$ and $\mu_0=m-s>0$. Since $J$ is differentiable at $0$, $J(\beta)<J(0)=q\mu_0^2/(2\sigma^2)$ for all sufficiently small $\beta>0$.
\end{proof}

\subsection*{Note:}
Claude Fable $5$ was used in the drafting of this note. This is the first result due to the Large Math Initiative, an initiative that aims to streamline the use of AI for math research.

\bibliographystyle{abbrvnat}
\bibliography{ref}
\end{document}